\documentclass[12pt,reqno,fleqn]{amsart}  
\usepackage{tikz}
\usepackage{amsmath,amstext,amsthm,amssymb,amsxtra}
\usepackage[top=1.5in, bottom=1.5in, left=1.25in, right=1.25in]	{geometry}
\usepackage{lmodern}

\usepackage{mathtools}
\mathtoolsset{showonlyrefs,showmanualtags}

\usepackage{hyperref} %,hypertexnames=false,colorlinks,[pagebackref]
\hypersetup{
	colorlinks=true,       % false: boxed links; true: colored links
	linkcolor=blue,          % color of internal links
	citecolor=magenta,        % color of links to bibliography
	filecolor=magenta,      % color of file links
	urlcolor=cyan           % color of external links
}

\usepackage[msc-links]{amsrefs}

\newtheorem*{theorem}{Theorem}

\newtheorem{lemma}{Lemma}

\newtheorem{OldTheorem}{Theorem}
\theoremstyle{definition}
\newtheorem*{definition}{Definition}
\theoremstyle{definition}

\theoremstyle{remark}
\newtheorem{remark}{Remark}[section]

\def\ZZ{\ensuremath{\mathbb Z}}

\def\ZN{\ensuremath{\mathbb N}}

\def\ZD{\ensuremath{\mathcal D}}
\def\ZK{\ensuremath{\mathcal K}}
\def\ZR{\ensuremath{\mathbb R}}

\def\ZD{\ensuremath{\mathcal D}}

\def\md#1#2\emd{\ifx0#1
	\begin{equation*} #2 \end{equation*}\fi  %  single line display, no number
	\ifx1#1\begin{equation}#2\end{equation}\fi   % single line display, number
	\ifx2#1\begin{align*}#2\end{align*}\fi   % aligned display, no number
	\ifx3#1\begin{align}#2\end{align}\fi    % aligned display, number
	\ifx4#1\begin{gather*}#2\end{gather*}\fi  % multline, not align, no number
	\ifx5#1\begin{gather}#2\end{gather}\fi   % multinline, not align
	\ifx6#1\begin{multline*}#2\end{multline*}\fi  %  display too long for one line
	\ifx7#1\begin{multline}#2\end{multline}\fi  % as above, with numbers
	\ifx8#1\begin{multline*}\begin{split}#2\end{split}\end{multline*}\fi
	\ifx9#1\begin{multline}\begin{split}#2\end{split}\end{multline}\fi
}
\newcommand {\e }[1]{\eqref{#1}}
\newcommand {\lem }[1]{Lemma \ref{#1}}

\newcommand {\trm }[1]{Theorem \ref{#1}}

\title[] {On a criterion of uniform distribution}
\author{Grigori A. Karagulyan}
\address{Institute of Mathematics of NAS of RA, Marshal Baghramian ave., 24/5, Yerevan, 0019, Armenia} 
\email{g.karagulyan@gmail.com}
\author{Iren A. Petrosyan}
\address{Faculty of Mathematics and Mechanics, Yerevan State
	University, Alex Manoogian, 1, 0025, Yerevan, Armenia} 
\email{iren.petrosyan1@edu.ysu.am}

\thanks{The work was supported by the Higher Education and Science Committee
	of RA, in the frames of the research project 21AG‐1A045}


\begin{document}
\begin{abstract}
We give an extension of a criterion of van der Corput on uniform distribution of sequences. Namely, we prove that a sequence $x_n$ is uniformly distributed modulo  1 if it is weakly monotonic and satisfies the conditions $\Delta^2x_n\to 0,\quad n^2\Delta^2x_n\to \infty $. Our proof is straightforward and uses a Diophantine approximation by rational numbers, while van der Corput's approach is based on some estimates of exponential sums.
\end{abstract}
	\maketitle  
\section{Introduction}
Let $X=\{x_n,\,n=1,2,\ldots\}$ be a sequence of real numbers. Denote by $A_N(X,I)$ the counting function that is the number of elements $x_n$, $n=1,2,\ldots,N$,  which fractional parts $\{x_n\}$ are contained in certain interval $I=[a,b)\subset [0,1)$. A sequence $X$ is said to be uniformly distributed modulo 1 if
	\begin{equation}
		\lim_{N\to\infty}\frac{A_N(X,I)}{N}=b-a
	\end{equation}
for every interval $I\subset [0,1)$. The classical Weyl criterion (\cite{KN}*{page Theorem 2.1}) states that
a sequence $\{x_n\}$ is uniformly distributed if and only if for any integer $h\neq 0$ we have 
\begin{equation}
	\lim_{N\to\infty}\frac{1}{N}\sum_{k=1}^N\exp(2\pi ihx_k)=0.
\end{equation}
A classical example of a uniformly distributed sequence is $\{n\theta\}$, where $\theta$ is an irrational number, that immediately follows from Weyl's criterion. In fact, Weyl's criterion is a powerful tool in the study of uniform distribution of sequences. It enables to reduce many problems in the uniform distribution theory to estimates of exponential sums. In many criteria of uniform distribution properties of the first and higher order differences of sequences play significant role. Those are defined inductively as follows:
\begin{equation}
	\Delta x_n=x_{n+1}-x_n,\quad \Delta^{k+1} x_n=\Delta(\Delta^{k} x_n),\quad k=1,2,\ldots.
\end{equation}
The following criteria are well-known in the theory of uniform distribution.
\begin{OldTheorem}[van der Corput, \cite{VDC4}]\label{VC1}
	If a sequence $\{x_n\}$ satisfies the condition $\Delta x_n\to \theta$, where
$\theta$ is an irrational number, then $\{x_n\}$ is uniformly distributed.
\end{OldTheorem}
\begin{OldTheorem}[van der Corput, \cite{VDC4}]\label{VC2}
	If a sequence $\{x_n\}$ satisfies the conditions 
	\begin{equation}
		\Delta^k x_n\searrow 0,\quad n\Delta^k x_n\to\infty \text{ as }n\to\infty
	\end{equation}
	for some $k\ge 1$, then it is uniformly distributed.
\end{OldTheorem}
\begin{OldTheorem}[van der Corput, \cite{VDC1}]\label{VC}
A sequence $\{x_n\}$ is uniformly distributed whenever
\begin{align}
	& \Delta x_n\to\infty,\quad\Delta^2 x_n \searrow 0,\\
	& \frac{n^2\Delta^2 x_n}{(\Delta x_n)^2}\to \infty\text{ as }n\to\infty.
\end{align}
\end{OldTheorem}
Theorems \ref{VC1} and \ref{VC2} can be proved using either Weyl's criterion or applying straightforward approach, while the proof of \trm{VC} follows from subtle quantitative estimates of certain exponential sums due to van der Corput \cite{VDC1} (see also \cite{KN}*{pp. 15-18}). In fact, van der Corput's method of exponential sums developed in papers \cite{VDC1, VDC2, VDC3, VDC4}  have had applications in many different number-theoretic problems (see \cite{GK, Kok, Tit}). In the present paper we give an extension of \trm {VC}, applying a different approach, without using trigonometric sums. Namely, we divide our sequence $x_n$ into segment groups $\{x_j:\, n_k<j\le n_{k+1}\}$ having $\varepsilon$-discrepancy, where integers $n_k$ satisfy the bound $n_k<n_{k+1}<n_k(1+\varepsilon)$ and are determined by an iteration as follows:  having $n_k$, we apply a suitable Diophantine approximation to the first difference $\Delta x_{n_{k}}$, then $n_{k+1}$ is defined depending on certain properties of this approximation. We use two key lemmas (Lemma \ref{L5} and \ref{L1}) in the estimations of discrepancies of segment groups. 
\begin{definition}
	A sequence of numbers $\{x_n\}$ is said to be weakly-decreasing if $x_n\ge 0$ and there exists a constant $\ZK\ge 1$ such that
	\begin{equation}
		\max_{j> k}x_j\le \ZK\min_{j\le k}x_j,\quad k=1,2,\ldots,
	\end{equation}
or equivalently $x_j\le \ZK x_n$ whenever $j>n$. We say that $\{x_n\}$ is weakly-increasing if $\{-x_n\}$ is weakly-decreasing.
\end{definition}
\begin{theorem}\label{T1}
	A sequence $\{x_n\}$ is uniformly distributed if $\Delta^2x_n$ is weakly-monotonic and satisfies the conditions 
	\begin{equation}\label{a}
		\Delta^2x_n\to 0,\quad n^2\Delta^2x_n\to \infty \text{ as }n\to\infty.
	\end{equation}
\end{theorem}
\begin{remark}
	The condition $n^2\Delta^2x_n\to \infty $ in the theorem is sharp, i.e. it may not be replaced by 
	\begin{equation}\label{x73}
		n^2\omega(n)\Delta^2x_n\to \infty 
	\end{equation}
	for any sequence $\omega(n)\to \infty$. Namely, the sequence $x_n=\log n$ is not uniformly distributed, but satisfies \e{x73} whenever $\omega(n)\to \infty$.
\end{remark}

\section{Notations and auxiliary lemmas}

\begin{lemma}\label{L3}
	Let $Y=\{y_1,y_2,\ldots,y_m\}$ be a sequence such that both $y_k$ and $\Delta y_k$ are monotonic. Then for any interval $[a,b]$ we have
	\begin{equation}\label{x27}
\# (Y\cap [a,b])\le \left(\frac{2(b-a)}{\min_k|\Delta^2 y_k|}\right)^{1/2}+2.
	\end{equation}
\end{lemma}
\begin{proof}
Without loss of generality we can suppose that both $y_k$ and $\Delta y_k$ are increasing and  so $\Delta^2 y_k\ge 0$. If $\# (Y\cap [a,b])\le 2$, then \e{x27} is immediate. So we can suppose $Y\cap [a,b]=\{y_{t+1},y_{t+2},\ldots,y_{t+s}\}$, $s\ge 3$. We have
\begin{align}
	b-a&\ge y_{t+s}-y_{t+1}\\
	&=\Delta y_{t+1}+\Delta y_{t+2}+\ldots+\Delta y_{t+s-1}\\
	&=(s-1)\Delta y_{t+1}+\sum_{j=1}^{s-2}(s-j-1)\Delta^2 y_{t+j}\\
	&\ge \frac{(s-2)(s-1)}{2}\cdot\min_k\Delta^2 y_k\\
	&\ge \frac{(s-2)^2}{2}\cdot \min_k\Delta^2 y_k,
\end{align}
which implies \e{x27}.
\end{proof}
\begin{lemma}\label{L6}
	If a sequence $Y=\{y_1,y_2,\ldots,y_m\}$ satisfies the conditions of \lem{L3}, then
	\begin{equation}\label{x26}
		\# Y=m\ge \left(\frac{2|y_m-y_1|}{\max \{\max_k|\Delta^2 y_k|,\min_k|\Delta y_k|\}}\right)^{1/2}.
	\end{equation}
\end{lemma}
\begin{proof}
Likewise to the proof of the previous lemma we can suppose that both $y_k$ and $\Delta y_k$ are increasing. So we have $\min_k|\Delta y_k|=\Delta y_1$. Thus we obtain
	\begin{align}
	y_m-y_1&=(m-1)\Delta y_1+\sum_{j=1}^{m-2}(m-j-1)\Delta^2 y_j\\
		&\le \frac{m(m-1)}{2}\cdot \max \{\max_k|\Delta^2 y_k|,\Delta y_1\}\\
		&\le \frac{m^2}{2}\cdot \max \{\max_k|\Delta^2 y_k|,\min_k|\Delta y_k|\}
	\end{align}
and so \e{x26} follows.
\end{proof}
\begin{lemma}\label{L7}
		Let a sequence $Y=\{y_1,y_2,\ldots,y_m\}$ be such that both $y_k$ and $\Delta y_k$ are increasing. Then for any two intervals $J=[a,b)$ and $I=[c,d)$ with $J<I$ (i.e. $b\le c$) and $J\subset [y_1,y_m)$ we have
		\begin{equation}\label{x33}
			\frac{\#(Y\cap I)-1}{|I|}\le \frac{\#(Y\cap J)+1}{|J|}.
		\end{equation}
\end{lemma}
\begin{proof}
	If $\#(Y\cap I)\le 1$, then \e{x33} is trivial. So we can suppose that $\#(Y\cap I)\ge 2$. One can check,
	\begin{equation}\label{x34}
		\frac{\#(Y\cap I)-1}{|I|}\le \frac{1}{\min_{[y_j,y_{j+1})\subset I}\Delta y_j}.
	\end{equation}
	If $\#(Y\cap J)=0$, then $J\subset [y_k,y_{k+1})$ for some $k$. Then we can write
	\begin{equation}
		|J|\le \Delta y_k\le \min_{[y_j,y_{j+1})\subset I}\Delta y_j
	\end{equation}
and combining this with \e{x34}, we obtain
\begin{equation}
	\frac{\#(Y\cap I)-1}{|I|}\le \frac{1}{|J|}=\frac{\#(Y\cap J)+1}{|J|}.
\end{equation}
If $\#(Y\cap J)=1$, then, applying the condition $J<I$ and that $\Delta y_k$ is increasing, we can write
\begin{equation}
	|J|\le 2\min_{[y_j,y_{j+1})\subset I}\Delta y_j.
\end{equation}
Thus, using also \e{x34}, we obtain
\begin{equation}
	\frac{\#(Y\cap I)-1}{|I|}\le \frac{2}{|J|}=\frac{\#(Y\cap J)+1}{|J|}.
\end{equation}
Hence, we can suppose $\#(Y\cap J)\ge 2$. One can check
\begin{equation}
	\frac{\#(Y\cap J)+1}{|J|}\ge \frac{1}{\max_{[y_j,y_{j+1})\cap J\neq \varnothing}\Delta y_j}\ge \frac{1}{\min_{[y_j,y_{j+1})\subset I}\Delta y_j}\ge \frac{\#(Y\cap I)-1}{|I|}.
\end{equation}
\end{proof}
Given a finite sequence $Y=\{y_1,y_2,\ldots,y_m\}$ and an interval $[a,b)\subset [0,1)$ denote by
\begin{align*}
	A(Y,[a,b))&=\#\{1\le k\le m:\, \{y_k\}\in [a,b)\}\\
	&=\sum_{j\in \ZZ}\#\{1\le k\le m:\, y_k\in [j+a,j+b)\}
\end{align*}
the number of elements of $Y$ including in the interval $[a,b)$ by modulo one. Define the discrepancy of the set $Y$ by 
\begin{align*}
	\ZD(Y)=\sup_{0\le a<b\le 1}\left|\frac{A(Y,[a,b))}{m}-(b-a)\right|.
\end{align*}
\begin{lemma} \label{L5}
	Let a sequence $Y=\{y_1,y_2,\ldots,y_m\}$, $m\ge 2$, satisfy the conditions of \lem{L3}.
	Then we have
	\begin{equation}\label{a12}
		\ZD(Y)\le 2\left(\frac{|y_m-y_1|}{m}+\frac{\max_{j\in \ZZ}\#(Y\cap [j,j+1))}{m}\right).
	\end{equation}
\end{lemma}
\begin{proof}
	We can suppose that both $y_k$ and $\Delta y_k$ are increasing. Choose $[a,b)\subset [0,1)$ and denote
	\begin{align}
		&M=\max_{k\in \ZZ}\#(Y\cap [k,k+1)),\label{x37}\\
		&I_j=[j+a,j+b),\quad  U_j=[j,j+1).\label{x38}
	\end{align}
	Let $[s,l)$ be the minimal integer-endpoints interval, containing the sequence $Y$.  It is clear 
	\begin{align}
		&l-s-2\le y_m-y_1,\label{x47}\\
		&U_j\subset [y_1,y_m),\quad s+1\le j\le l-2.
	\end{align}
Observe that $J=U_{j-1}$ and $I=I_j$ satisfy the conditions of \lem{L7}. Thus, from \e{x33} we conclude
	\begin{align}
		\#(Y\cap I_j)&\le 	(b-a)(\#(Y\cap U_{j-1})+1)+1\\
		&\le (b-a)\#(Y\cap U_{j-1})+2,\quad s+2\le j\le l-1\label{x46}
	\end{align}
	and similarly
	\begin{equation}
		\#(Y\cap I_j)\ge (b-a)\#(Y\cap U_{j+1})-2,\quad s+1\le j\le l-2.
	\end{equation}
	Applying \e{x47} and \e{x46}, we obtain
	\begin{align}
		\#\{1\le k\le m:\, &\{y_k\}\in [a,b)\}=\sum_{j=s}^{l-1}\#(Y\cap I_j)\\
		&\le (b-a)\sum_{j=s+2}^{l-1}\#(Y\cap U_{j-1})+2(l-s-2)\\
		&\qquad \qquad +\#(Y\cap I_{s})+\#(Y\cap I_{s+1})\\
		&\le(b-a)m+2(y_m-y_1)+\#(Y\cap I_{s})+\#(Y\cap I_{s+1})
	\end{align}
and therefore,
\begin{align}
&\frac{\#\{1\le k\le m:\, \{y_k\}\in [a,b)\}}{m}-(b-a)\\
&\qquad\qquad\qquad\le \frac{2(y_m-y_1)}{m}+\frac{2M}{m}.\label{x35}
\end{align}
Likewise, we can prove
\begin{align}
	&\frac{\#\{1\le k\le m:\, \{y_k\}\in [a,b)\}}{m}-(b-a)\\
	&\qquad\qquad\qquad\ge -\frac{2(y_m-y_1)}{m}-\frac{2M}{m},\label{x36}
\end{align}
and combining \e{x35} and \e{x36}, we get \e{a12}.
\end{proof}
\begin{remark}
	Observe that under the conditions of \lem{L5} we have 
\begin{equation}
		|y_m-y_1|=|\Delta y_1+\Delta y_2+\ldots+\Delta y_{m-1}|\le (m-1)\max_{k}|\Delta y_k|.
\end{equation}
So from \e{a12} we get
\begin{equation}\label{x68}
	\ZD(Y)\le 2\left(\max_k|\Delta y_k|+\frac{\max_{j\in \ZZ}\#(Y\cap [j,j+1))}{m}\right).
\end{equation}
\end{remark}
	In the next lemma we will use the following standard numerical inequality 
	\begin{equation}\label{x49}
		\sum_{k=1}^na_k\cdot \sum_{k=1}^nb_k\le n\sum_{k=1}^na_kb_k,
	\end{equation}
where $a_k$  and $b_k$, $k=1,2,\ldots,n$, are decreasing positive sequences. 

\begin{lemma}\label{L1}  Let $Y=\{y_1,y_2,\ldots,y_m\}$, $m\ge 3$, be a sequence such that $y_k$ is increasing and $\Delta^2 y_k$ is weakly-decreasing with a constant $\ZK$, then
\begin{equation}\label{x41}
\ZD(Y)\lesssim\frac{ y_m-y_1}{m}+\frac{\ZK}{\sqrt{y_m-y_1}}.
\end{equation}
\end{lemma}
\begin{proof}
Without loss of generality we can suppose that $y_1=0$. From the conditions of the lemma it also follows that $\Delta y_k$ is positive and increasing.
According to \lem{L5}, it is enough to show that 
\md1\label{a3}
\#(Y\cap U_j)\lesssim \frac{m\ZK}{\sqrt {y_m}},\text{ for all } j=0,1,\ldots,l=[y_m],
\emd
where $U_j=[j,j+1)$. Since both $y_k$ and $\Delta y_k$ are increasing, from \lem{L7} it follows that $\#(Y\cap U_j)\le \#(Y\cap U_0)+2$ and so
\begin{equation*}
	\max_{j\in \ZZ}\#(Y\cap U_j)\le \#(Y\cap U_0)+2\le 3 \#(Y\cap U_0).
\end{equation*}
Hence, we need to show that 
\begin{equation}\label{x39}
	\#(Y\cap U_0)\lesssim \frac{m\ZK}{\sqrt {y_m}}.
\end{equation}
For $k\ge 2$ we have
\begin{align}
	y_k&=\Delta y_1+\Delta y_2+\ldots+\Delta y_{k-1}\\
	&=(k-1)\Delta y_1+(k-2)\Delta^2 y_1+\ldots+\Delta^2 y_{k-2}\label{x40}
\end{align}
(if $k=2$, then the last sum is simply $\Delta y_1$). Without loss of generality we can suppose that
\begin{equation}\label{x42}
	\Delta y_1<1/3,\quad \Delta^2 y_1<1/3.
\end{equation}
Indeed, if one of these bounds doesn't hold, we would have $y_m\gtrsim m$ according to \e{x40} and the condition $m\ge 3$. So bound \e{x41} will be trivially satisfied, since for the left hand side we always have $\ZD(Y)\le 1$.
Denote
\begin{equation}
	\mu_0=\max\{\Delta y_1,\Delta^2 y_1\},  \mu_k= \Delta^2 y_k,\quad k\ge 1.  
\end{equation}
Since $\Delta^2 y_k$ is weakly-decreasing, so we have for the sequence $\mu_0,\mu_1,\ldots$ with the same constant $\ZK$ as $\Delta^2 y_k$ has. Consider the modified sequence 
\begin{equation}
	z_1=y_1=0,\quad z_k=(k-1)\mu_0+(k-2)\mu_1+\ldots+\mu_{k-2},\quad k\ge 2.
\end{equation}
It is easy to verify that
\begin{equation}
	y_k\le z_k\le 3y_k,\quad k\ge 3.
\end{equation}
Now suppose that
\begin{equation*}
	Y\cap U_0=\{y_1,y_2,\ldots,y_p\}
\end{equation*}
that means $y_{p}<1\le y_{p+1}$. From \e{x42} it follows that $p\ge 3$ and since $\mu_j$ is weakly decreasing we can write
\md1\label{x50}
z_{p}\ge \mu^*_{p-2}\sum_{j=0}^{p-2}(p-j-1)\ge \frac{p(p-1)}{2\ZK}\cdot \mu_{p-2},
\emd
where $\mu_k^*=\min_{j\le k}\mu_j$. Besides, applying \e{x49}, we get
\begin{equation}\label{x51}
	\sum_{j=0}^{p-2}\mu_j\le\ZK\sum_{j=0}^{p-2}\mu_j^*\le\frac{2\ZK }{p}\sum_{j=0}^{p-2}(p-j-1) \mu_j^*\le\frac{2\ZK }{p}\sum_{j=0}^{p-2}(p-j-1) \mu_j.
\end{equation}
Therefore, applying \e{x50} and \e{x51}, we get
\md2
z_m&=\sum_{j=0}^{m-2}(m-j-1)\mu_j=\sum_{j=0}^{p-2}(m-j-1)\mu_j+\sum_{j=p-1}^{m-2}(m-j-1)\mu_j\\
&=\sum_{j=0}^{p-2}(p-j-1)\mu_j+(m-p)\sum_{j=0}^{p-2}\mu_j+\sum_{j=p-1}^{m-2}(m-j-1)\mu_j\\
&\le \sum_{j=0}^{p-2}(p-j-1)\mu_j+\frac{2\ZK(m-p)}{p}\sum_{j=0}^{p-2}(p-j-1)\mu_j+\ZK\mu_{p-2}\sum_{j=p-1}^{m-2}(m-j-1)\\
&\le z_{p}+\frac{2\ZK mz_{p}}{p}+\frac{2\ZK^2 z_{p}}{p(p-1)}\cdot \frac{m(m-1)}{2}\lesssim \ZK^2 \left(\frac{m}{p}\right)^2,
\emd
where we also used the facts that $p\ge 3$ and $z_p\le 3y_p<3$. This immediately implies 
\begin{equation}
\#(Y\cap U_0)=p\lesssim \frac{m}{\sqrt{z_m/\ZK^2}}\le  \frac{m\ZK}{\sqrt{y_m}}
\end{equation}
that gives \e{x39}, completing the proof of lemma.
\end{proof}

\begin{lemma}\label{L2}
	Let sequences  $X=\{x_k:\,k=1,2,\ldots, m\}$ and $Y=\{y_k:\,k=1,2,\ldots, m\}$ satisfy $|x_k-y_k|<\varepsilon$. Then
	\begin{equation*}
		\ZD(Y)\le \ZD(X)+2\varepsilon.
	\end{equation*}
\begin{proof}
	Suppose $\ZD(X)=\delta$ and $I=[a,b)\subset [0,1)$ is an arbitrary interval. We need to show
	\begin{equation}\label{x52}
		\left|\frac{A(Y,I)}{m}-|I|\right|\le \delta+2\varepsilon.
	\end{equation}
Consider the $I^*=[a-\varepsilon,b+\varepsilon)$ and let $I^{**}=\{\{x\}:\, x\in I^*\}$.
We have 
\begin{align}
	&|I^{**}|= \min\{|I^*|,1\}\le |I^*|,\\ 
	&\cup_{k\in \ZZ}(Y\cap (k+I))\subset \cup_{k\in \ZZ}(X\cap (k+I^*)=\cup_{k\in \ZZ}(X\cap (k+I^{**}).
\end{align}
So, using the definition of discrepancy, we obtain 
\begin{align}
	A(Y,I)&=\#\left(\cup_{k\in \ZZ}(Y\cap (k+I))\right)\le\#\left(\cup_{k\in \ZZ}(X\cap (k+I^*))\right)\\
	&= A(X,I^{**})\le m(|I^*|+\delta )=m(|I|+\delta+2\varepsilon).\label{x53}
\end{align}
 To show the lower bound first suppose that $|I|>\delta+2\varepsilon$ and consider the interval $I_*=[a+\varepsilon, b-\varepsilon)$.
One can similarly show
\begin{equation}\label{x54}
A(Y,I)\ge m(|I|-\delta-2\varepsilon).
\end{equation}
It remains just note that in the case of $|I|\le \delta+2\varepsilon$ the last inequality trivially holds. Combining \e{x53} and \e{x54}, we get \e{x52}, completing the proof of lemma.
\end{proof}
\end{lemma}
\begin{lemma}\label{L4}
	Let $\{x_k\}$ be an infinite sequence and $\varepsilon >0$. If for any $n\ge n_0$ there is an integer $m\le n\varepsilon $ such that 
	\begin{equation}\label{x23}
		\ZD\{x_k:\, n<k\le n+m\}\le \varepsilon,
	\end{equation}
then 
\begin{equation}\label{x78}
	\limsup_{N\to\infty }\ZD\{x_k:\, 1\le k\le N\}\le 2\varepsilon.
\end{equation}
\end{lemma}
\begin{proof}
According to the conditions of lemma, we can find a sequence of integers $n_0<n_1<n_2\ldots<n_k<\ldots$ such that $n_j<n_{j+1}\le (1+\varepsilon) n_j$ and 
\begin{equation*}
		\ZD\{x_k:\, n_j<k\le n_{j+1}\}\le \varepsilon.
\end{equation*}
We claim 
\begin{equation}\label{x76}
	\limsup_{N\to\infty }\frac{\#\{1\le k\le N:\, \{x_k\}\in [a,b)\}}{N}\le b-a+2\varepsilon
\end{equation}
for any interval $[a,b)\subset [0,1)$. If $b-a+\varepsilon\ge  1$, then \e{x76} trivially holds. So we can assume that $b-a+\varepsilon\le 1$ and for $n_m<N\le n_{m+1}$ we can write
\begin{align}
\#\{1\le k\le N:\, \{x_k\}\in [a,b)\}&\le n_0+\sum_{j=0}^{m}\#\{\{x_k\}\in [a,b):\, n_j<k\le n_{j+1}\}\\
&\le n_0+ (b-a+\varepsilon)\sum_{j=0}^{m}(n_{j+1}-n_j)\\
&\le  n_0+(b-a+\varepsilon)N+(n_{m+1}-n_m)\\
&\le  n_0+(b-a+\varepsilon)N+n_m\varepsilon\\
&<  n_0+(b-a+\varepsilon)N+N\varepsilon,
\end{align}
that implies \e{x76}. To prove the lower bound
\begin{equation}\label{x77}
	\liminf_{N\to\infty }\frac{\#\{1\le k\le N:\, \{x_k\}\in [a,b)\}}{N}\ge  b-a-2\varepsilon,
\end{equation}
we can suppose that $b-a-\varepsilon\ge 0$, then \e{x77} follows from 
\begin{align}
	\#\{1\le k\le N:\, \{x_k\}\in [a,b)\}&\ge \sum_{j=0}^{m-1}\#\{\{x_k\}\in [a,b):\, n_j<k\le n_{j+1}\}\\
	&\ge (b-a-\varepsilon)\sum_{j=0}^{m-1}(n_{j+1}-n_j)\\
	&\ge  (b-a+\varepsilon)n_m\\
	&\ge  \frac{(b-a-\varepsilon)N}{1+\varepsilon}\\
	&\ge (b-a-2\varepsilon)N.
\end{align}
Combining \e{x76} and \e{x77}, we get \e{x78}.
\end{proof}
\begin{lemma}\label{L8}
	Let $Y_k\subset \ZR$, $k=1,2,\ldots,m$, be finite sequences and $Y$ be a sequence obtained by an arbitrary ordering of the elements of the sequences $Y_k$. Then,
	\begin{equation}\label{x55}
		\ZD(Y)\le \max_k\ZD(Y_k).
	\end{equation}
\end{lemma}
\begin{proof}
	Let $[a,b)\subset [0,1)$ be an arbitrary interval. For a given $m\in \ZN$ we can write
	\begin{align*}
		\left|\frac{A(Y,[a,b))}{\# Y}-(b-a)\right|&=\left|\frac{\sum_kA(Y_k,[a,b))}{\#Y}-(b-a)\right|\\
		&\le \sum_k\frac{\#Y_k}{\#Y}\left|\frac{A(Y_k,[a,b))}{\#Y_k}-(b-a)\right|\\
		&\le \sum_k\frac{\#Y_k}{\#Y}\ZD(Y_k)\\
		&\le \max_k\ZD(Y_k).
	\end{align*}
This implies \e{x55}.
\end{proof}
\section{Proof of Theorem}

\begin{proof}
Without loss of generality we can suppose that $\Delta^2 x_n$ is weakly-decreasing. Let 
\begin{equation}\label{x72}
	0<\varepsilon<1/10
\end{equation}
be fixed. Using \e{a}, we find an integer $n(\varepsilon)$ such that
	\begin{equation}\label{x5}
	\frac{1}{\varepsilon^8 n^2}\le \Delta^2x_n<\varepsilon^{12} \text{ for all }n>n(\varepsilon)>\varepsilon^{-5}.
\end{equation}
Then we fix $n>n(\varepsilon)$. According to \lem{L4} it is enough to prove that there exists an integer $m\le 2n\varepsilon$, such that
\begin{equation}\label{x24}
	\ZD\{x_k:\, n<k\le n+m\}\lesssim \varepsilon.
\end{equation}
 Suppose that $p_k/q_k$ is the sequence of continued fraction approximation of $\Delta x_n$. Namely, there  are sequences of coprime integers $p_k$ and $q_k$ such that
\begin{align}	
	&\left|\Delta x_n-\frac{p_{k}}{q_{k}}\right|\le \frac{1}{q_{k}q_{k+1}},\label{a7}\\
	&1=q_0<q_1<q_2<\ldots
\end{align}
(see for example \cite{Buc} chap. 24 or \cite{EiTh} chap. 3). We may find a unique integer $k=k(n)$ such that $q_{k}\le \varepsilon^{-4}<q_{k+1}$. Denote
$p=p_k$, $q=q_{k}$, $q'=q_{k+1}$. Hence we have
\begin{align}
	&\left|\Delta x_n-\frac{p}{q}\right|\le \frac{1}{qq'},\label{a6}\\
	&q\le\varepsilon^{-4}<q'.\label{a5}
\end{align}
To prove \e{x24} we will consider two possible cases of $q$. Besides, for the second case three different sub-cases will be discussed.

{\bf Case 1: $q>\varepsilon^{-1}$.} For any $0\le j\le q$ we have
\begin{equation*}
	\Delta x_{n+j}=\Delta x_n+\sum_{i=n}^{n+j-1}\Delta^2 x_i
\end{equation*}
and so, using \e{x5}, \e{a6} and \e{a5}, we obtain
\begin{align}
	\left|\Delta x_{n+j}-\frac{p}{q}\right|&\le  \left|\Delta x_n-\frac{p}{q}\right|+\sum_{i=n}^{n+q-1}\Delta^2 x_i\le \frac{1}{qq'}+q\varepsilon^{12}\\
	&\le\frac{1}{q^2}+q\cdot \frac{1}{q^3}=\frac{2}{q^2}.
\end{align}
Therefore,
\begin{align}
	\left|x_{n+j}-x_n-\frac{jp}{q}\right|&=\left|\sum_{i=n}^{n+j-1}\left(\Delta x_i-\frac{p}{q}\right)\right|\\
	&\le \sum_{i=n}^{n+j-1}\left|\Delta x_i-\frac{p}{q}\right|\le  \frac{2}{q}\label{x25}
\end{align}
for all $1\le j\le q$.  Thus, using assumption $q>\varepsilon^{-1}$, for the sequences
\begin{align}
	&Y=\{y_j=x_{n+j}-x_n:\, j=1,2,\ldots ,q\},\\
	&Z=\left\{z_j=\frac{jp}{q}:\, j=1,2,\ldots ,q\right\}
\end{align}
we have $|y_j-z_j|<2/q<2\varepsilon$ and, applying \lem{L2}, we get 
\begin{equation}
	\ZD(Y)\le \ZD(Z)+4\varepsilon.
\end{equation}
On the other hand, since $p$ and $q$ are coprime integers, we have
\begin{equation*}
	\bar Z=\left\{\left\{\frac{jp}{q}\right\}:\, j=1,2,\ldots ,q\right\}=	\left\{\frac{s}{q}:\, s=0,1,\ldots ,q-1\right\}
\end{equation*}
and so $\ZD(Z)=\ZD(\bar Z)=1/q<\varepsilon$. Thus, we obtain
\begin{equation}
	\ZD\{x_{n+j}:\, j=1,2,\ldots ,q\}=\ZD(Y)\le \ZD(Z)+4\varepsilon<5\varepsilon.
\end{equation}
Since $q\le n\varepsilon $ (see \e{x5}, \e{a5}), we get \e{x24} with $m=q$.

{\bf Case 2: $q\le \frac{1}{\varepsilon} $.} 
We claim there is an integer $m\le 2n\varepsilon^2$ such that 
\begin{equation}\label{x15}
	\ZD\{x_{n+r},x_{n+r+q},\ldots,x_{n+r+(m-1)q}\}\lesssim \varepsilon,\quad m\le 2n\varepsilon^2,
\end{equation}
for all $1\le r\le q$. Having \e{x15}, we may obtain \e{x24}. Indeed, applying \lem{L8}, from \e{x15} it follows that
\begin{align}
	\ZD\{x_{n+1},x_{n+2},&\ldots,x_{n+qm}\}\\
	&\le \max_{1\le r\le q}\ZD\{x_{n+r},x_{n+r+q},\ldots,x_{n+r+(m-1)q}\}\lesssim \varepsilon.
\end{align}
It remains just note that $mq\le 2q n\varepsilon^2\le 2n\varepsilon$. 
In order to prove \e{x15} consider the sequences
\md0
y_k(r)=-kp+x_{n+r+(k-1)q},\quad k=1,2,\ldots,
\emd
for the parameters $1\le r\le q$. Observe that
\begin{equation}\label{x69}
	\ZD\{x_{n+r},x_{n+r+q},\ldots,x_{n+r+(m-1)q}\}=\ZD\{y_k(r):\, k=1,2,\ldots,m\}.
\end{equation}
Let us state some properties of the sequences $y_k(r)$.

P1) First, observe that
\begin{align}
	\Delta y_k(r)&=y_{k+1}(r)-y_k(r)=-p+x_{n+r+kq}-x_{n+r+(k-1)q}\\
	&=-p+\sum_{j=0}^{q-1}\Delta x_{n+r+(k-1)q+j}\\
	&=-p+\sum_{j=0}^{q-1}\left(\Delta x_{n}+\sum_{i=n}^{n+r+(k-1)q+j-1}\Delta^2x_{i}\right)\\
	&=-p+q\Delta x_{n}+\sum_{j=0}^{q-1}\,\,\sum_{i=n}^{n+r+(k-1)q+j-1}\Delta^2x_{i}\\
	&=\alpha+\sum_{j=0}^{q-1}\,\,\sum_{i=n}^{n+r+(k-1)q+j-1}\Delta^2x_{i}\label{x21},
\end{align}
where
\begin{equation}\label{a11}
	|\alpha|=|q\Delta x_n-p|<\frac{1}{q'}<\varepsilon^{-4} \text{ (see }\e{a6},\e{a5}).
\end{equation}

P2) From \e{x5} and \e{x21} it follows that
\begin{align}
	\Delta^2 y_k(r)&=\sum_{j=0}^{q-1}\,\,\sum_{i=n+r+(k-1)q+j}^{n+r+kq+j-1}\Delta^2x_i\label{x22}\\
	&\ge \frac{q^2}{(n+(k+2)q)^2\varepsilon^{8}}\gtrsim \frac{1}{n^2\varepsilon^8} \text{ if }k\le 3n. \label{x74}
\end{align}
Since $\Delta^2 x_k$ is weakly-decreasing, one can easily check that so is $\Delta^2 y_k(r)$ with the same constant $\ZK$. Moreover, we have
\begin{align}
	& \Delta^2 y_k(r)\le q^2\varepsilon^{12}\le \varepsilon^{10},\quad k=1,2,\ldots,\text{(see }\e{x5}),\label{x70}\\
	& \Delta^2 y_k(r)\le \ZK\Delta^2 y_s(r')\text{ if }r'+sq\le r+kq,\, 1\le r,r'\le q.\label{x67}\\
\end{align}

P3) From \e{x21} we can also see that
\begin{align}
	&\Delta y_k(r')\le \Delta y_k(r)\le \Delta y_{k+1}(r'), \quad 0\le r'\le r\le q,\label{x56}\\
	&\Delta y_{k+1}(r)=\Delta y_k(r+q).\label{x57}
\end{align}
Then, using also the representations
\begin{align*}
	&y_k(r)-y_s(r)=\Delta y_s(r)+\Delta y_{s+1}(r)+\ldots+\Delta y_{k-1}(r),\\
	&y_k(0)-y_s(0)=\Delta y_s(0)+\Delta y_{s+1}(0)+\ldots+\Delta y_{k-1}(0),
\end{align*}
we conclude
\begin{equation}\label{x62}
	y_k(0)-y_s(0)\le y_k(r)-y_s(r)\le y_{k+1}(0)-y_s(0)-\Delta y_s(0).
\end{equation}

P4) Note that all the terms in \e{x21} are non-negative except $\alpha$, which can be either positive or negative. 
Besides, using \e{x56}, one can check that  $\Delta y_k(r)\nearrow \beta$, where $\beta$ doesn't depend on $r$ ($\beta$ can also be infinite). In the case $\beta>0$ we may find an integer $h(r)\ge 0$ such that 
\begin{align}
	\Delta y_k(r)\le 0&\hbox{ if } k\le  h(r),\label{x28}\\
		\Delta y_k(r)\ge 0&\hbox{ if }k> h(r).\label{x32}
\end{align}
If $\beta\le 0$, then we set $h(r)=\infty$. From \e{x56} and \e{x57}, we can conclude that $h(r)$ is non-increasing in $r$ and $h(0)=h(q)+1$. So we have 
\begin{equation}\label{x58}
	h(0)\ge h(r)\ge h(0)-1,\quad 1\le r\le q.
\end{equation}
Taking into account \e{x21} and \e{a11}, so we can write
\begin{equation}\label{x18}
	|\Delta y_k(r)|<|\alpha| <\varepsilon^4,\text{ whenever }k\le h(r).
\end{equation}

\vspace{5mm}
We will consider the following three possible sub-cases of Case 2, applying some restrictions to the number $h(0)$.

{\bf Case 2.1: $h(0)>n\varepsilon^2$.} Choose $m=[n\varepsilon^2]-1<n/2$ (see \e{x72}). From \e{x58} we get $h(r)>m$. So according to \e{x28} the sequence $\{y_k(r):\, k=1,2,\ldots,m\}$ satisfies the conditions of \lem{L3}, namely, $\Delta y_k(r)$ is increasing and negative. 
On the other hand, using \e{x74}, we can write
\begin{equation}\label{x44}
	\min_{1\le k\le m}\Delta^2 y_k(r)\gtrsim \frac{1}{n^2\varepsilon^8}.
\end{equation}
First, applying  \lem{L3}, we obtain
\begin{equation}\label{x16}
	\#(\{y_1(r),y_2(r),\ldots, y_m(r)\}\cap [k,k+1])\lesssim \frac{1}{\sqrt{(n^2\varepsilon^8)^{-1}}}= n\varepsilon^4.
\end{equation}
Then, applying \lem{L5} (see also \e{x68}) together with \e{x69} and \e{x18}, we conclude
\begin{align}
		\ZD\{x_{n+r},x_{n+r+q},\ldots,x_{n+r+(m-1)q}\}&=\ZD\{y_k(r):\, k=1,2,\ldots,m\}\\
		&\lesssim \max_{1\le j< m}|\Delta y_j(r)|+\frac{n\varepsilon^4}{m}\lesssim \varepsilon^{4} +\varepsilon^2\lesssim\varepsilon
\end{align}
that gives \e{x15} . 

{\bf Case 2.2:} $1/\varepsilon\sqrt{\delta}<h(0)\le n\varepsilon^2$, where 
\begin{equation}\label{x71}
	\delta=\Delta^2 y_{h(0)-1}(0)\lesssim \varepsilon^{10} \text{ (see } \e{x70}).
\end{equation}
In this case we will choose $m=h(0)-2$. Using the weakly-decreasing property of $\Delta^2 y_k(r)$ and relation \e{x67}, we can write 
\begin{align}
	&\min_{1\le k\le m}\Delta^2 y_k(r)\ge \frac{\Delta^2 y_{m}(r)}{\ZK}\ge \frac{\Delta^2 y_{m+1}(0)}{\ZK^2}\\
	&\qquad\qquad\qquad\qquad\qquad\qquad=\frac{\Delta^2 y_{h(0)-1}(0)}{\ZK^2}= \frac{\delta}{\ZK^2}
\end{align}
By \e{x28} we have $m<h(r)$. Therefore the sequence $\{y_1(r),y_2(r),\ldots, y_m(r)\}$ satisfies the conditions of \lem{L3}, since $\Delta y_{k}(r)$ is increasing and  by \e{x58} $\Delta y_{k}(r)\le 0$ for $k=1,2,\ldots,m$.  Thus, applying \lem {L3}, we conclude
\begin{equation}
	\#(\{y_1(r),y_2(r),\ldots, y_m(r)\}\cap [k,k+1])\lesssim \frac{1}{\sqrt{\delta}}.
\end{equation}
Once again applying \lem{L5} (with \e{x68}), \e{x18} as well as the assumption 
\begin{equation}
	m=h(0)-2>1/\varepsilon\sqrt{\delta}-2\gtrsim 1/\varepsilon\sqrt{\delta},
\end{equation} 
we get
\begin{align}
\{y_k(r):\, &k=1,2,\ldots,m\}\\
	&\lesssim \varepsilon^4+\frac{\max_{k\in \ZZ}\#(\{y_1(r),y_2(r),\ldots, y_m(r)\}\cap [k,k+1])}{m}\lesssim\varepsilon
\end{align}
and so \e{x15}.

{\bf Case 2.3: $h(0)\le \min\left\{\frac{1}{\varepsilon\sqrt{\delta}},n\varepsilon^2\right\}$.}
Observe that 
\begin{equation}\label{x59}
	h(0)-1\le h(r)\le h(0)\le \min\left\{\frac{1}{\varepsilon\sqrt{\delta}},n\varepsilon^2\right\} \text{ (see }\e{x58})
\end{equation}
and the sequence 
\begin{equation}\label{x61}
	Y=\{y_{k}(r):\,h(0)+1\le k\le  l=2[n\varepsilon^2]\}
\end{equation}
satisfies the conditions of Lemmas \ref{L3} and \ref{L6}, since both $y_{k}(r)$ and $\Delta y_{k}(r)$ are increasing in this range according to \e{x32} and  \e{x58}. We know that the sequence $\Delta^2 y_k$ is weakly-decreasing, so the sequence $Y$ satisfies also conditions of \lem{L1}. Besides, we have
\begin{align}
	&\min_{h(0)+1\le  k\le l}|\Delta^2y_k(r)|\gtrsim\frac{ 1}{n^2\varepsilon^8}, \,(\text{see }\e{x74})\label{x29}\\
	&\max_{h(0)+1\le  k\le l}|\Delta^2y_k(r)|\le\ZK\Delta^2 y_{h(0)-1}(0)=\ZK \delta,(\text{see }\e{x67})\label{x30}\\
	&\min_{k\ge h(0)+1}\Delta y_k(r)= \Delta y_{h(0)+1}(r)\\
	&\qquad\qquad\qquad\quad=\Delta y_{h(0)-1}(r)+\Delta^2 y_{h(0)-1}(r)+\Delta^2 y_{h(0)}(r)\\
	&\qquad\qquad\qquad\quad\le\Delta^2 y_{h(0)-1}(r)+\Delta^2 y_{h(0)}(r)\, (\text{see }  \e{x28}, \e{x59})\\
	&\qquad\qquad\qquad\quad\le\ZK(\Delta^2 y_{h(0)-1}(0)+\Delta^2 y_{h(0)}(0))\, (\text{see }  \e{x69})\\
	&\qquad\qquad\qquad\quad\le \ZK(\ZK+1)\Delta^2 y_{h(0)-1}(0)=\ZK(\ZK+1)\delta,\\
	& \qquad\qquad\qquad\qquad \text{ whenever }k\ge h(0)+1.\label{x31}
\end{align}
Using  \lem{L3} with $[a,b]=[y_{h(r)+1}(r),y_{l}(r)],$ as well as bounds \e{x59} and \e{x29}, we obtain
\begin{equation}
	y_{l}(r)-y_{h(r)+1}(r)\gtrsim (l-h(r))^2\cdot \frac{1}{n^2\varepsilon^8}\gtrsim (n\varepsilon^2)^2\cdot \frac{1}{n^2\varepsilon^8}= \varepsilon^{-4}.
\end{equation}
Then since sequence \e{x61} is increasing, we may find an integer $m\le l$ such that
\begin{equation}\label{x63}
	y_{m-1}(0)-y_{h(0)+1}(0)\lesssim \varepsilon^{-4}\lesssim y_{m}(0)-y_{h(0)+1}(0).
\end{equation}
From \e{x62} and \e{x63} we get
\begin{align}
	&y_{m-2}(r)-y_{h(0)+1}(r)\le y_{m-1}(0)-y_{h(0)+1}(0)-\Delta y_{h(0)+1}(0)\\
	&\qquad\qquad  \le  y_{m-1}(0)-y_{h(0)+1}(0)\lesssim   \varepsilon^{-4},\label{x64}\\
	& y_{m}(r)-y_{h(0)+1}(r) \ge  y_{m}(0)-y_{h(0)+1}(0)\gtrsim \varepsilon^{-4}.\label{x65}
\end{align}
From \e{x64} and \e{x70} it follows that 
\begin{align}
	&\Delta y_{m-3}(r)\le \Delta y_{h(0)+1}(r)+\Delta y_{h(0)+2}(r)+\ldots+\Delta y_{m-3}(r)\\
	&\qquad\qquad=y_{m-2}(r)-y_{h(0)+1}(r)\lesssim \varepsilon^{-4},\\
	&\Delta y_{m-2}(r)\le \Delta y_{m-1}(r)=\Delta y_{m-3}(r)+\Delta^2 y_{m-3}(r)+\Delta^2 y_{m-2}(r)\lesssim \varepsilon^{-4}.\label{x75}
\end{align}
Thus, using \e{x64} and \e{x75}, we obtain
\begin{equation}\label{x66}
	y_{m}(r)-y_{h(0)+1}(r)=y_{m-2}(r)-y_{h(0)+1}(r)+\Delta y_{m-2}(r)+\Delta y_{m-1}(r)\lesssim \varepsilon^{-4}.
\end{equation}
From bounds \e{x30} and \e{x31} we obtain
\begin{equation}
\max \left\{	\max_{h(0)+1\le  k\le l}|\Delta^2y_k(r)|,\min_{h(0)+1\le k\le l}|\Delta y_k(r)|\right\}\lesssim\delta.
\end{equation}
Thus, applying \lem{L6} together with \e{x65},  we obtain
\begin{equation}
	m^2\delta\ge  (m-h(0))^2\delta\gtrsim y_m(r)-y_{h(0)+1}(r)\ge \varepsilon^{-4}
\end{equation}
and then, combining the restriction $h(r)\le \frac{1}{\varepsilon\sqrt{\delta}}$ (see \e{x59}), we obtain
\begin{equation}\label{x43}
	m\ge \frac{1}{\varepsilon^2\sqrt{\delta}}>\frac{h(r)}{\varepsilon}\gtrsim \frac{h(0)}{\varepsilon} .
\end{equation}
Applying \lem{L1} and bounds \e{x71}, \e{x65}, \e{x66}, we get
\begin{align*}
	\ZD\{y_{h(0)+1}(r),y_{h(0)+2}(r),\ldots, y_{m}(r)\}&\lesssim \frac{y_{m}(r)-y_{h(0)+1}(r)}{m}+\frac{1}{\sqrt{y_{m}(r)-y_{h(0)+1}(r)}}\\
	&\lesssim \varepsilon^{-4}\cdot \varepsilon^2\sqrt\delta+\varepsilon^2\lesssim \varepsilon.
\end{align*}
Finally, applying \e{x43}, we obtain
\begin{align}
	\ZD\{y_1(r),y_{2}(r),&\ldots, y_{m}(r)\}\\
	&\le \frac{h(0)}{m}+	\ZD\{y_{h(0)+1}(r),y_{h(0)+2}(r),\ldots, y_{m}(r)\}\lesssim \varepsilon.
\end{align}
Taking into account \e{x69} and that $m\le l\le 2n\varepsilon^2$, this completes the proof of \e{x15} and so the theorem.
\end{proof}

\end{document}